\documentclass[12pt]{article}
\usepackage[utf8]{inputenc}
\usepackage{amsthm}
\usepackage{amsmath}

\makeatletter
\def\@seccntformat#1{\csname the#1\endcsname.\ } 
\makeatother

\newtheorem{lemma}{Lemma}
\newtheorem{theorem}{Theorem}
\newtheorem{proposition}{Proposition}
\newtheorem{corollary}{Corollary}
\theoremstyle{definition}
\newtheorem{remark}{Remark}

\title{The Steiner triple systems of order 21 with a transversal subdesign TD(3,6)%
\thanks{This research is supported by
the National Natural Science Foundation of China (61672036),
Excellent Youth Foundation of Natural Science Foundation of Anhui Province (No.1808085J20),
and the 
Program of Fundamental Scientific Research of the SB RAS No I.5.1., project No 0314-2019-0016.}
}
\author{%
 Yue Guan%
 \thanks{Y. Guan is with the School of Mathematical Sciences, Anhui University, Hefei, 230601, China}%
,
 Minjia Shi%
 \thanks{M. Shi is with the School of Mathematical Sciences, Anhui University, Hefei, 230601, China; corresponding author, e-mail: smjwcl.good@163.com}
,
Denis S. Krotov%
\thanks{D. Krotov is with the Sobolev Institute of Mathematics, Novosibirsk, 630090, Russia}%
}
\date{}

\def\aA{\mathcal A}
\def\bB{\mathcal B}
\def\cC{\mathcal C}
\def\dD{\mathcal D}
\def\tT{\mathcal T}
\def\gG{\mathcal G}
\def\pP{\mathcal P}

\def\STS{{\rm STS}}
\def\TD{{\rm TD}}

\begin{document}

\maketitle

\begin{abstract}
We prove several structural properties of Steiner triple systems (STS) of order $3w+3$ that include one or more transversal subdesigns TD$(3,w)$. Using an exhaustive search, we find that there are $2004720$ isomorphism classes of STS$(21)$ including a subdesign TD$(3,6)$, or, equivalently, a $6$-by-$6$ latin square.


\end{abstract}

\section{Introduction}

A \emph{Steiner triple system} of order $v$, or \emph{\STS$(v)$}, is a pair $(S,\bB)$ from a finite set $S$ (called the support, or the point set, of the \STS)
of cardinality $v$ and a collection $\bB$ of $3$-subsets of $S$, called \emph{blocks}, such that
every two distinct elements of $S$ meet in exactly one block.
A transversal design \TD$(k,w)$ (in this paper, we only consider the case $k=3$) is a triple $(S, \gG, \bB)$ that consists of a point set $S$ of cardinality $kw$, a partition $\gG$ of $S$ into $k$ subsets, \emph{groups}, of cardinality $w$, and a collection $\bB$ of $k$-subsets of $S$, \emph{blocks}, such that every block intersects every group in exactly one point and every two points in different groups
meet in exactly one block.
As the support and (in the case of \TD) the groups are uniquely determined by the block set, it is convenient to identify the system, STS or TD, with its block set.
With this agreement, it is correct to say that an STS $\bB$ can include, as a subset,
some STS or TD $\cC$, in which case $\cC$ is called
a \emph{sub-STS} or \emph{sub-TD} of $\bB$, respectively.
Two systems, STS or TD, are called \emph{isomorphic}
if there is a bijection between their supports,
an \emph{isomorphism},
that sends the blocks of one system to the blocks of the other.
An isomorphism of a system $\bB$ to itself is called an \emph{automorphism};
the set of all automorphisms
of $\bB$ is denoted by $\mathrm{Aut}(\bB)$.

Transversal designs \TD$(3,w)$ are equivalent to latin $w \times w$ squares and known to exist for every natural $w$.
The isomorphism classes of \TD$(3,w)$ correspond to the so-called main classes of latin squares;
their number is known for $w$ up to $11$, see \cite{HKO:2011:latin}.
Steiner triple systems \STS$(v)$   exist if and only if $v\equiv 1,3 \bmod 6$, see, e.g., \cite{ColMat:Steiner},
the necessary condition being given by simple counting arguments.

The number of isomorphism classes of Steiner triple system is known for order up to $19$ \cite{KO:STS19}.
The classification of STS of higher orders is possible only
with additional restrictions on the structure of STS.
Among such restrictions,
the most popular are restrictions on the automorphisms,
see e.g.
\cite{Bays:23}, \cite{Colbourn:79}, \cite{Denniston:80}, \cite{MPR:81}, \cite{PhRo:81}, \cite{Tonchev:87},
\cite{KapTop:92},
restrictions on the maximal rank of the system
\cite{Osuna:2006}, \cite{JMTW:STS27},
requirement for the system to include a subsystem with certain
fixed parameters.

Stinson and Seah \cite{StiSea:1986}
found that there are $284457$
\STS$(19)$ with sub-\STS$(9)$.
Kaski,
\"Osterg{\aa}rd,
Topalova, and
Zlatarski
\cite{KOTZ:2008}
classified \STS$(19)$
with sub-\STS$(7)$
and \STS$(21)$
that include three sub-\STS$(7)$
with disjoint supports
(the last class coincides with the class
of \STS$(21)$ of $3$-rank at most $19$).
Recently, Kaski, \"Osterg{\aa}rd, and Popa
\cite{KOP:2015:subSTS}
counted all \STS$(21)$ with sub-\STS$(9)$
(and also, the \STS$(27)$ with sub-\STS$(13)$).
The number $ 12661527336 $
(respectively, $1356574942538935943268083236$)
of isomorphism classes of such systems
is too large to admit any constructive enumeration; in particular, one cannot
computationally check any required property
for all these classes.

In the current paper, we classify the \STS$(21)$ with subdesigns \TD$(3,6)$, or saying
in a different way, the \STS$(21)$ that include a latin $6\times 6$ square.
We establish that there are $2004720$ isomorphism classes of
Steiner triple systems of order $21$ with transversal subdesigns on $3$ groups of size $6$,
including $599$ systems with exactly three sub-TD$(3,6)$ and $12$
systems with exactly seven sub-TD$(3,6)$.
Considered class contains $393$ non-isomorphic resolvable STS;
none of them is doubly-resolvable.

In the next section, we prove some facts about the Steiner triple systems of order $3w+3$
that have a transversal subdesign on three groups of size $w$, mainly focused on the case $w=6$.
In Section~\ref{s:calc}, we present the results of computer-aided classification of \STS$(21)$
with a subdesign \TD$(3,6)$, including Table~\ref{t:1},
which contains the number of found isomorphism classes classified by the number of subdesigns
\TD$(3,6)$, \STS$(9)$, and the number of automorphisms.
Section~\ref{s:valid} contains a double-counting argument that validates the results of computing.
In Section~\ref{s:resol}, we discuss the resolvability of the found STS and show that
STS$(21)$ with sub-\TD$(3,6)$ and only one sub-\STS$(9)$ cannot be resolvable.

\section{Steiner triple systems with transversal subdesigns}\label{s:theory}
We start with some theoretical considerations.
If an \STS$(v)$ has a sub-\TD$(3,w)$, then $v=3w+u$,
where $u\equiv 1,3$ if $w$ is even and $u\equiv 0,4$ if $w$ is odd.
The case $u=0$ corresponds to the Wilson-type \STS$(3w)$ \cite{Wilson:74:STS};
readily, such a system is the union of
 three \STS$(w)$ with mutually disjoint supports
 and a transversal design  \TD$(3,w)$.
The case $u=1$ corresponds to the Wilson-type \STS$(3w+1)$ \cite{Wilson:74:STS};
again, it is easy to see that such a system is a union of
three \STS$(w+1)$ whose supports have one point in common
and a transversal design \TD$(3,w)$.

The next case is $u=3$. We introduce a related concept.
A subset $\cC$ of an STS $\bB$ is called an \emph{almost-sub-STS} if
$\cC=\cC'\backslash \{T\}$
for some STS $\cC'$ and a triple $T$ of $\cC'$ (note that $T$ is not required to be a block of $\bB$);
this triple is called \emph{missing} for the almost-sub-STS  $\cC$.

\begin{lemma}\label{l:3w+3}
Let $(A\cup B \cup C\cup D, \mathcal B)$ be an \STS$(3w+3)$ with a sub-TD
$(A\cup B \cup C, \{A, B, C\}, \mathcal T)$,
where $|A|=|B|=|C|=w$ and $|D|=3$.
Then $\mathcal B = \mathcal T \cup \mathcal B_A \cup \mathcal B_B \cup \mathcal B_C$,
where the supports of $\mathcal B_A$, $\mathcal B_B$, $\mathcal B_C$ are $A\cup D$, $B\cup D$, $C\cup D$
respectively, and two of them are almost-sub-STS with the missing triple $D$,
the remaining one being a sub-STS (in particular, $w+3\equiv 1,3 \bmod 6$).
\end{lemma}
\begin{proof}
 Assume first that $D$ is one of the blocks of $\bB$.
 In this case, it is easy to see that $\bB$ has a sub-\STS$(w+3)$ with support $A\cup D$.
 Take it as $\bB_A$.
 Similarly, $\bB$ has a sub-\STS$(w+3)$ with support $B\cup D$.
 Removing the block $D$, we obtain an almost-sub-STS $\bB_B$.
 Similarly, we find an almost-sub-STS $\bB_C$.

 It remains to consider the case $D\not\in\bB$. In this case, we divide $\bB\backslash \tT$
 into three subsets: $\bB_A$ consists of the blocks that are subsets of $A\cup D$,
 $\bB_B$ of the blocks that are subsets of $B\cup D$,
  and $\bB_C$ of  subsets of $C\cup D$ (note that any other block
  has points in at least two of $A$, $B$, $C$, and hence necessarily belongs to $\tT$).
  The blocks from $\bB_A$, $\bB_B$, and $\bB_C$ cover
  \[
\begin{split}
&\frac12|A|\cdot(|A|-1)+\frac12|B|\cdot(|B|-1)+\frac12|C|\cdot(|C|-1)+3|A|+3|B|+3|C|+3\\
&= \frac 32 w^2+\frac{15}2 w +3
\end{split}
\]
  pairs of points,
  while the blocks from $\bB_A$ (similarly, from $\bB_B$ or from $\bB_C$)
  cover at least
  $$\frac12|A|\cdot(|A|-1)+3|A| =\frac 12 w^2+\frac{5}2 w$$
  and at most
  $$\frac12|A|\cdot(|A|-1)+3|A|+3 = \frac 12 w^2+\frac{5}2 w+3$$
  of them.
  Since the number of the pairs covered by $\bB_A$
  must be divisible by $3$, we see that it is
  $$\frac 12 w^2+\frac{5}2 w\quad \mbox{or } \frac 12 w^2+\frac{5}2 w+3\quad \mbox{if } w\equiv 0,1\bmod 3,$$
  and it is $$\frac 12 w^2+\frac{5}2 w+2 \quad\mbox{if }w\equiv 2\bmod 3.$$
  The last case is impossible because  $3(\frac 12 w^2+\frac{5}2 w+2)>\frac 32 w^2+\frac{15}2 w +3$.
  We conclude that one of $\bB_A$, $\bB_B$, $\bB_C$, say $\bB_A$,
  has
  $$\frac 13\cdot \Big(\frac 12 w^2+\frac{5}2 w +3\Big) = \frac{(w+3)(w+2)}6$$
  blocks, while each of the other has one less. This means that
  $( A\cup D, \bB_A)$ is an \STS$(w+3)$ and
  $( A\cup D, \bB_B\cup\{D\})$, $( A\cup D, \bB_B\cup\{D\})$
  are also \STS$(w+3)$, which proves the statement.
\end{proof}

So, we see that if an \STS$(3w+3)$ has a sub-\TD$(3,w)$, then it is split into this sub-\TD$(3,w)$,
one sub-\STS$(w+3)$, and two almost-sub-\STS$(w+3)$. In general, there can be more than one sub-\TD$(3,w)$
and hence more than one such splittings. So, it is important to understand how these subsystems can intersect. The following lemma on the intersection of two almost-sub-STS generalizes the well-known
and obvious fact that the intersection of two sub-STS is always a sub-STS.

\begin{lemma}\label{l:2sub}
Assume that an STS $\bB$ has two almost-sub-STS $\bB'$ and $\bB''$
with the supports $S'$ and $S''$, respectively. Then
\begin{itemize}
 \item either $|S'\cap S''|=2$ and
$\bB' \cap \bB''$ is empty
 \item  or $|S'\cap S''| \equiv 1,3 \bmod 6$ and $\bB' \cap \bB''$ can be completed to \STS$(|S'\cap S''|)$ by adding
 $0$, $1$, or $2$ blocks.
\end{itemize}
\end{lemma}
\begin{proof}
Denote $D:=S'\cap S''$ and $\dD:=\bB' \cap \bB''$.
 Assume that 
 $\bB' \cup \{a',b',c'\}$ 
 is an STS and 
 $\bB' \cup \{a'',b'',c''\}$ 
 is an STS.
 So, any pair of points from $S'$ except $\{a',b'\}$,  $\{a',c'\}$,  $\{b',c'\}$ is included
 in one block of $\bB'$. Similarly, with $S''$ and $\bB''$.
 Therefore,
 \begin{itemize}
 \item[(*)] \emph{Any pair of points from $D$ different from
  $\{a',b'\}$,  $\{a',c'\}$,  $\{b',c'\}$,
   $\{a'',b''\}$,  $\{a'',c''\}$,  $\{b'',c''\}$
 is included in one block of $\dD$.}
\end{itemize}
Further,
\begin{itemize}
 \item[(**)] \emph{For every point $t$ from $D$, 
 the number $l(t)$ of pairs of points from $D$ containing
 $t$ and different from $\{a',b'\}$,  $\{a',c'\}$,  $\{b',c'\}$,
$\{a'',b''\}$,  $\{a'',c''\}$,  $\{b'',c''\}$ is even.} 
Indeed, this number is twice the number of blocks of $\dD$ containing $t$.
\end{itemize}
By easy check of all cases for the intersections
of the sets $D$, $\{a',b',c'\}$, and $\{a'',b'',c''\}$, 
one can find that if both (*) and (**) are satisfied, then one of the following assertions holds.

\begin{enumerate}
 \item[(i)] $D$ contains at most one point from $a'$, $b'$, $c'$
 and at most one point from $a''$, $b''$, $c''$. In this case, $\dD$ is an STS.
 \item[(ii)] $D$ consists of two  points from $a'$, $b'$, $c'$ or
 from $a''$, $b''$, $c''$. So, $|D|=2$.
  \item[(iii)] $D$ contains $a'$, $b'$, $c'$ and at most one of $a''$, $b''$, $c''$;
  or $D$ contains $a'$, $b'$, $c'$ and two or three of $a''$, $b''$, $c''$ that coincide with some of $a'$, $b'$, $c'$. Or, analogously,  $D$ contains $a''$, $b''$, $c''$ and at most one of $a'$, $b'$, $c'$;
  or $D$ contains $a''$, $b''$, $c''$ and two or three of $a'$, $b'$, $c'$ that coincide with some of $a''$, $b''$, $c''$. In this case, $\dD$ can be completed to an STS by adding the triple
  $\{a',b',c'\}$ or $\{a'',b'',c''\}$, respectively.
    \item[(iv)] $D$ includes both $\{a'$, $b'$, $c'\}$ and $\{a''$, $b''$, $c''\}$
    and these two sets intersect in at most one point. In this case, $\dD$ can be completed to an STS by adding the two triples  $\{a',b',c'\}$ and $\{a'',b'',c''\}$.
    \item[(v)] $|D|=4$ and two points $e'$, $f'$ from $D$ are in  $\{a',b',c'\}$, while the other two $e''$, $f''$ are in $\{a'',b'',c''\}$.
    It is easy to see that the four pairs
    $\{e',e''\}$, $\{e',f''\}$, $\{f',e''\}$, $\{f',f''\}$ cannot be covered by blocks of $\dD$ in a proper way.
\end{enumerate}
We have been convinced that the statement of the lemma holds in any non-contradictory case.
\end{proof}

\begin{remark}
In contrast to the case of sub-STS's,
the supports of two almost-sub-STS's can intersect in exactly
two points. One can easily construct such example using known embedding theorems:
any set of $3$-sets such that no pair of points meets in more than one set can always be embedded
as a subset in a Steiner triple system, whose support can in general be larger that the support
of the original triple set \cite{Linder:survey-embed}, \cite{BryHor:Lindner}.
\end{remark}

\begin{corollary}\label{c:2sub}
 {\rm (i)} Different supports of two almost-sub-\STS$(9)$ of the same STS intersect in at most $3$ points.
 {\rm (ii)} Different supports of two almost-sub-\STS$(9)$ of the same \STS$(21)$ intersect in $3$ points. These
 $3$ points form either a block of each of the two almost-sub-\STS$(9)$,
 or the missing triple of one or both of the almost-sub-\STS$(9)$.
\end{corollary}
\begin{proof}
 To prove (i), by Lemma~\ref{l:2sub}, it remains to verify that the supports
 cannot intersect in $7$ points. Indeed, if such situation happens, then
 by Lemma~\ref{l:2sub}, there are \STS$(7)$ and \STS$(9)$ that have at least $5$ blocks in common.
 It is straightforward to check that this is not possible.

 (ii) It remains to prove that the supports, say $S'$ and $S''$, of the subsystems, say $\bB'$ and $\bB''$, cannot intersect in $2$, $1$, or $0$ points.
 Assume first that $|S' \cap S''|=2$.
 The block including $S' \cap S''$ does not belong to at least one
 of $\bB'$ and $\bB''$.
 W.l.o.g., assume that it is not in $\bB''$; so, the pair $S' \cap S''$ lies in the missing triple
 of the almost-sub-STS $\bB''$.
 Hence, $6$ points $a_1$, \ldots, $a_6$ from $S'' \backslash S'$
 do not belong to the missing triple. Take also one point $b$ from $S' \backslash S''$ that
 do not belong to the missing triple of $\bB'$ and consider the block containing $b$ and $a_i$, $i\in \{1,2,3,4,5,6\}$.
 This block is not from $\bB'$ or $\bB''$; hence, it intersects with each of $S'$ and $S''$
 in only one point. So, the third point $c_i$ of this block does not belong to $S'\cup S''$.
 But there are only $5$ points not in $S'\cup S''$, which immediately leads to a contradiction
 with the definition of STS.
 Similar contradictions can be found in the cases $|S' \cap S''|=1$ and $|S' \cap S''|=0$.
\end{proof}

Next, we focus on the order $21$. Assume we are given an \STS$(21)$ $(S,\bB)$.
A partition of $S$ into four sets $A$, $B$, $C$, $D$ of size $6$, $6$, $6$, and
and $3$ respectively is called a \emph{flower} with   \emph{stem} $D$ and   \emph{petals} $A$, $B$, $C$
if $\bB$ has a sub-\STS$(9)$ and two almost-sub-\STS$(9)$ with supports $A\cup D$, $B\cup D$, $C\cup D$,
where the missing triple of each of these almost-sub-STS is $D$
(whenever it belongs to $\bB$ or not). 
From Lemma~\ref{l:3w+3}, 
we can easily deduce the following.
\begin{lemma}\label{l:TD-flow}
An \STS$(21)$ has a flower $\{A$, $B$, $C$, $D\}$ with the stem $D$ if and only if it has a sub-\TD$(3,6)$
with groups $A$, $B$, $C$.
\end{lemma}

If there is only one flower $\{A$, $B$, $C$, $D\}$ (and only one sub \TD$(3,6)$), then we have two subcases, depending on whether
the stem $D$ is a block or not.
In the first subcase, the \STS$(21)$ has three sub-\STS$(9)$ with supports $A\cup D$, $B\cup D$, $C\cup D$.
In the second subcase, the \STS$(21)$ has only one sub-\STS$(9)$.
Our next goal is to characterize the situation when \STS$(21)$ has more than one sub-\TD$(3,6)$.

\begin{lemma}\label{l:2flow}
Assume that an \STS$(21)$ has two different flowers  $\{A,B,C,D\}$ and $\{A',B',C',D'\}$ with stems $D$, $D'$. Then
\begin{itemize}
 \item[\rm (i)] $D$ and $D'$ are disjoint;
 \item[\rm (ii)] $D\cup E=D' \cup E'$, for some $E\in\{A,B,C\}$ and $E'\in\{A',B',C'\}$;
\item[\rm (iii)] the \STS$(21)$ has a sub-\STS$(9)$ with support $D\cup E$, where $E$ is from p.(ii).
\end{itemize}
\end{lemma}
\begin{proof}
Consider a point $d$ from $D$ and assume without loss of generality that it lies in $D'\cup A'$.
Every point of $D'\cup A'$ lies in one of $D\cup A$, $D\cup B$, $D\cup C$ and at least one point lies
in all. Hence,   $D'\cup A'$ intersects in more than $3$ points with these three sets in average.
By Corollary~\ref{c:2sub}, it coincides with one of them, say $D\cup A$; so, (ii) is proved.
If $D$ and $D'$ intersect, then the same can be said about $D'\cup B'$ and $D'\cup C'$, and the flowers coincide. So, (i) holds.
The last claim is also easy as the union of two different almost-sub-STS with the same support $D\cup A$ is necessarily a sub-STS.
\end{proof}

\begin{lemma}\label{l:3flow}
Assume that an \STS$(21)$ $\bB$ has two different flowers
$\{A,B,C,D\}$ and $\{A',B',C',D'\}$ with stems $D$, $D'$.
Assume without loss of generality that $D\cup A=D' \cup A'$.
Denote
  $$A_{001}:=D, 
  \qquad\qquad\qquad 
  A_{011}:=A\backslash D', 
  \qquad \qquad\qquad 
  A_{010}:=D', $$
 $$A_{101}:=B\cap C', 
 \qquad
  \raisebox{ 2mm} {$A_{111}:=C\cap C',$} 
  \qquad
  A_{110}:=C\cap B',$$
   $$A_{100}:=B\cap B'.$$
The following assertions hold.
\begin{itemize}
 \item[\rm (i)]
If both $D$ and $D'$ are blocks of $\bB$,
then $\bB$ includes exactly $7$ sub-\TD$(3,6)$
with flowers
\begin{eqnarray}
\{A_{001},\ A_{010}\cap A_{011},\ A_{100}\cap A_{101},\ A_{110}\cap A_{111}\},\label{eq:A0} \\
\{A_{010},\ A_{001}\cap A_{011},\ A_{100}\cap A_{110},\ A_{101}\cap A_{111}\},\label{eq:A1} \\
\{A_{011},\ A_{001}\cap A_{010},\ A_{100}\cap A_{111},\ A_{101}\cap A_{110}\},\label{eq:A2} \\
\{A_{100},\ A_{001}\cap A_{101},\ A_{010}\cap A_{110},\ A_{011}\cap A_{111}\},\label{eq:A3} \\
\{A_{101},\ A_{001}\cap A_{100},\ A_{010}\cap A_{111},\ A_{011}\cap A_{110}\},\label{eq:A4} \\
\{A_{110},\ A_{001}\cap A_{111},\ A_{010}\cap A_{100},\ A_{011}\cap A_{101}\},\label{eq:A5} \\
\{A_{111},\ A_{001}\cap A_{110},\ A_{010}\cap A_{101},\ A_{011}\cap A_{100}\},\label{eq:A6}
\end{eqnarray}

and exactly $7$ sub-\STS$(9)$.
 \item[\rm (ii)]
If at most one of $D$, $D'$ is a block of $\bB$,
then $\bB$ includes exactly $3$ sub-\TD$(3,6)$
with flowers \eqref{eq:A0}--\eqref{eq:A2}.
In this case, if $D$ or $D'$ is a block,
then $\bB$ has exactly $3$ sub-\STS$(9)$;
otherwise exactly $1$.
\end{itemize}
\end{lemma}
\begin{proof}
We first note that by Corollary~\ref{c:2sub}(ii), $A_{101}$, $A_{111}$, $A_{110}$, and $A_{100}$
are blocks of $\bB$.
Next, we state that

(*) \emph{there is an almost-sub-STS with the support
$A_{011} \cap A_{100} \cap A_{111}$ and the missing triple $A_{011}$.}
Indeed, consider a block containing a point $a$ from $A_{100}$ and a point $b$ from $A_{111}$.
The third point $c$ of this block can only belong to $A_{011}$ (for example, if $c\in A_{001}$,
then the pair $\{a,c\}$ is already covered by a block from the almost-STS on $A_{001} \cap A_{101} \cap A_{100}$;
the other cases lead to similar contradictions). So, the $9$ such blocks form
a \TD$(3,3)$; completing by the blocks $A_{100}$ and $A_{111}$, we get an almost-sub-STS$(9)$.
Similarly,

(**) \emph{there is an almost-sub-STS with the support
$A_{011} \cap A_{101} \cap A_{110}$ and the missing triple $A_{011}$.}

So, we have a collection from a sub-\STS$(9)$
and six almost-sub-\STS$(9)$ with different supports,
corresponding to the flowers \eqref{eq:A0}--\eqref{eq:A2}.
It is easy to find that

(***) \emph{there is no sub-\STS$(9)$ or almost-sub-\STS$(9)$ with any other support.}
Indeed,
if $B$ is the support of a sub-\STS$(9)$, then it intersects in at least four points in total
with some two sets $A_{...}$;
the union of these two sets is included in the support of some of the seven sub-\STS$(9)$,
and by Corollary~\ref{c:2sub} $B$ coincides with this support.

Now consider subcases.

If both $D$ and $D'$ are in $\bB$, then we also have $A_{011} \in \bB$, and all those six almost-sub-STS
are completed to a sub-\STS$(9)$, forming seven different flowers in total.
From Lemma~\ref{l:3flow} and Corollary~\ref{c:2sub},
we conclude that there are no more flowers.
By (***), there are only seven sub-\STS$(9)$.

If $D\not\in \bB$ or $D'\not\in\bB$,
then \eqref{eq:A3}--\eqref{eq:A6} are not flowers.
Arguments similar as above show that there are only
three flowers \eqref{eq:A0}--\eqref{eq:A2}.

If $D\in \bB$ and $D'\not\in\bB$,
then we also have $A_{011} \not \in \bB$
(in any \STS$(9)$, the complement of two disjoint blocks is necessarily a block too).
In this case, we have only three sub-\STS$(9)$ with supports
$A_{001} \cup A_{011} \cup A_{010}$,
$A_{001} \cup A_{101} \cup A_{100}$,
$A_{001} \cup A_{111} \cup A_{110}$.
The subcase $D\not\in \bB$ and $D' \in \bB$ is similar.

 If $D\not\in \bB$ and $D'\not\in\bB$,
 then the missing triple of any of the six almost-sub-STS is not in $\bB$,
 and  $\bB$ has only one sub-\STS$(9)$, with the support $A_{001} \cup A_{011} \cup A_{010}$.
\end{proof}

\begin{remark}
 One can observe that each of the seven supports of (almost)-sub-\STS$(9)$ considered in the 
 lemma above is the union of three of $A_{001}$, $A_{010}$, $A_{011}$, $A_{100}$, $A_{101}$, $A_{110}$, $A_{111}$. The corresponding seven triples form an \STS$(7)$ on the point set  
 $\{A_{001}$, $A_{010}$, $A_{011}$, $A_{100}$, $A_{101}$, $A_{110}$, $A_{111}\}$
(the STS$(7)$ is unique up to isomorphism and known as the \emph{Fano plane}).
\end{remark}

The next two well-known and straightforward facts will be utilized in our further discussion.
\begin{proposition}\label{p:td6}
If a \TD$(3,6)$ $(S,\{A,B,C\},\tT)$ has a sub-\TD$(3,3)$ with the groups
$A_0\subset A$, $B_0\subset B$, $C_0\subset C$, then $\tT$ has exactly three other sub-\TD$(3,3)$,
with groups $A_0$, $B_1$, $C_1$, with groups $A_1$, $B_0$, $C_1$, and  with groups $A_1$, $B_1$, $C_0$,
where $A_1:=A\backslash A_0$,  $B_1:=B\backslash B_0$,  $C_1:=C\backslash C_0$.
\end{proposition}
\begin{proposition}\label{p:sts9}
If $D$ is a block of an \STS$(9)$ $(S,\bB)$,
then $\bB$ has exactly two blocks disjoint with $D$.
Moreover, these two blocks are disjoint with each other,
and the remaining $9$ blocks form a sub-\TD$(3,3)$.
\end{proposition}

\begin{lemma}\label{l:2td}
 Assume that an \STS$(21)$ $(S,\bB)$
 has a flower $\{A,B,C,D\}$, and $\tT$ is a transversal subdesign of $\bB$
 on the petals $A$, $B$, $C$, as the groups. Let $D'$ be a $3$-subset of $A$.
 The system $\bB$ has a second sub-\TD$(3,6)$ $\tT'$ with the support $S\backslash D'$
 if and only if it has
 disjoint blocks $B_0, B_1 \subset B$ and
 disjoint blocks $C_0, C_1 \subset C$ such that $\tT$
 is partitioned into four sub-\TD$(3,3)$ with groups from
 $D'$, $A\backslash D'$, $B_0$, $B_1$, $C_0$, $C_1$.
\end{lemma}
\begin{proof}
 Assume that there is such subdesign $\tT'$.
 In this case, there is a flower $\{D',A',B',C'\}$, where $D'\cup A'=D\cup A$.
 Denote $B_0:=B\cap B'$, $B_1:=B\cap C'$, $C_0:=C\cap B'$, $C_1:=C\cap C'$.
 By Corollary~\ref{c:2sub},  $B_0$, $B_1$, $C_0$, $C_1$ are blocks of $\bB$.
 Since $\bB$ has a sub-STS with the support $D'\cup A'$,
 by the definition of a flower it has two almost-sub-STS with the supports $D'\cup B'$
 and  $D'\cup C'$. Removing the blocks $B_0$, $B_1$, $C_0$, $C_1$
 from these almost-sub-STS's, we obtain two \TD$(3,3)$.
 The remaining two sub-\TD$(3,3)$ of
 $\tT$ are guaranteed by Proposition~\ref{p:td6}.

 The ``if'' part of the statement is also straightforward, taking into account
 Lemma~\ref{l:TD-flow}. If $B$ is partitioned into two blocks $B_0$, $B_1$,
 then, by the definition of a flower and Proposition~\ref{p:sts9}, we see that
 there is an almost-sub-STS with the support $B\cup D$ and the missing triple $D$
 (which can be a block or not a block of $\bB$). The same can be said about the support $C\cup D$.
 Then the definition of a flower implies that $\bB$ has a complete sub-STS with the support $A\cup D$.
 It remains to find two more petals, to form a flower with the stem $D'$.
 By the hypothesis, we have a sub-TD with groups $D'$, $B_0$ and $C_0$, for some block $C_0\subset C$.
 Completing it by the blocks $B_0$ and $C_0$, we get an almost-sub-STS.
 Similarly, we find an almost-sub-STS with the support $D'\cup B_1\cup C_1$, $C_1:=C\backslash C_0$.
 So, $\{D', D \cup A\backslash D', B_0\cup C_0, B_1\cup C_1\}$ is a flower, and by Lemma~\ref{l:TD-flow}
 there is a required sub-TD.
\end{proof}


\section{Classification of \STS$(21)$ with sub-\TD$(3,6)$.}\label{s:calc}

Now, based on the lemmas above, we are ready to present the way of the computer-aided classification and its results.
We start with describing how to count the number of isomorphism classes
of \STS$(21)$ with a unique sub-\TD$(3,6)$.

We first fix a flower $\{A_{001},\ A_{010}\cap A_{011},\ A_{100}\cap A_{101},\ A_{110}\cap A_{111}\}$,
where all sets $A_{...}$ are of size $3$.
Let $\mathbf A$ be the set of all $840$ \STS$(9)$ on $A_{001} \cap A_{010} \cap A_{011}$.
Denote by $\mathbf A'$ the subset of $\mathbf A$ that consists of
$120$ \STS$(9)$ with block $A_{001}$; by deleting this block in all these STS
we obtain the set $\mathbf A^*$ of $120$ almost-STS$(9)$ with missing $A_{001}$.
Denote by $\mathbf A''$ the subset of $\mathbf A'$ that consists of
$12$ \STS$(9)$ with blocks $A_{001}$, $A_{010}$, $A_{011}$.
Similarly, we define the collections $\mathbf B$, $\mathbf B'$, $\mathbf B''$, $\mathbf B^*$
of triple systems on $A_{001} \cap A_{100} \cap A_{101}$ and
the collections $\mathbf C$, $\mathbf C'$, $\mathbf C''$, $\mathbf C^*$
of triple systems on $A_{001} \cap A_{110} \cap A_{111}$.

Next, we choose a representative $\tT$ of one of $12$ (see \cite{A003090}) isomorphism classes
of \TD$(3,6)$ with groups $A_{010}\cap A_{011}$, $A_{100}\cap A_{101}$, $A_{110}\cap A_{111}$.
Moreover, we require that if the representative is divided into sub-\TD$(3,3)$'s, then
these sub-\TD's have the group sets
$\{A_{010}, A_{100}, A_{110}\}$,
$\{A_{010}, A_{101}, A_{111}\}$,
$\{A_{011}, A_{100}, A_{111}\}$,
$\{A_{011}, A_{101}, A_{110}\}$ (see Proposition~\ref{p:td6}).

Now, by Lemma~\ref{l:TD-flow}, every \STS$(21)$ with sub-TD $\tT$ is divided into
$\tT $, $ \aA$, $\bB$, and $\cC$,
where
\begin{itemize}
\item either
$\aA \in \mathbf A'$,\qquad $\bB \in \mathbf B^*$,\qquad $\cC \in \mathbf C^*$,
\item or
$\aA \in \mathbf A\backslash \mathbf A'$,\qquad $\bB \in \mathbf B^*$,\qquad $\cC \in \mathbf C^*$,
\item or
$\aA \in \mathbf A^*$,\qquad $\bB \in \mathbf B\backslash \mathbf B'$,\qquad $\cC \in \mathbf C^*$,
\item or
$\aA \in \mathbf A^*$,\qquad $\bB \in \mathbf B^*$,\qquad $\cC \in \mathbf C\backslash \mathbf C'$.
\end{itemize}
Moreover, by Lemmas~\ref{l:3flow} and~\ref{l:2td}, such \STS$(21)$ has exactly $7$ sub-\TD$(3,6)$ if and only if $\tT$ is divided into  sub-\TD$(3,3)$
and
\begin{eqnarray}\label{eq:ABC}
\aA \cup A_{001} \in \mathbf A'', \qquad
\bB \cup A_{001} \in \mathbf B'', \qquad 
\cC \cup A_{001} \in \mathbf C'';
\end{eqnarray}
and it has exactly $3$ sub-\TD$(3,6)$ if and only if $\tT$ is divides into  sub-\TD$(3,3)$
and
exactly two of \eqref{eq:ABC} are satisfied.
We exclude these cases
and finally have
at most $120^3+3\cdot 720\cdot 120^2$ \STS$(21)$ with only one sub-\TD$(3,6)$, equal to $\tT$.
Using the graph-isomorphism software \cite{nauty2014}, we can check all of them on isomorphism and keep
the representatives. Trivially, any STS$(21)$ that has a sub-TD isomorphic to $\tT$
is isomorphic to some STS$(21)$ that includes $\tT$.
Repeating the steps above for each of $12$ nonisomorphic choices of $\tT$, we find all equivalence classes
of \STS$(21)$ with only one sub-\STS$(9)$.

Similarly, we can classify the \STS$(21)$ with $3$ or $7$ sub-\STS$(9)$.
The only difference is that we also need to check
for isomorphism between the representatives
obtained from different $\tT$.

\begin{table}
\newcommand\ttss[2]{$\stackrel{\mbox{$\tau_{6}=#1$}}{\sigma_9=#2}$}
\begin{tabular}{c||c|c|c|c|c|c|c|}
$|\mathrm{Aut}|$ & \ttss{7}{7}  & \ttss33 & \ttss31 & \ttss13 & \ttss11 & \ttss01 \\
\hline\hline
1   &       & 98 (0) & 171 (0)        &  101621 (355) &  1865036 (0)     &  12656035473     \\  
2   &       & 45 (0) &  36 (0)        &  5271 (14)    &  30771  (0)      &  3461498         \\  
3   &       & 37 (0) &  66 (0)        &  103 (8)      &  52     (0)      &  14932           \\  
4   &       & 18 (0) &  14 (0)        &  321 (1)      &  786    (0)      &  10328           \\  
6   & 1 (1) & 31 (0) &  45 (0)        &  24 (1)       &  8      (0)      &  157             \\  
8   & 1 (0) & 7  (0) &   1 (0)        &  60 (5)       &  23     (0)      &  130             \\  
9   & 1 (1) &        &   9 (0)        &               &                  &  12              \\  
12  & 1 (1) & 6  (0) &   8 (0)        &  5 (0)        &  5      (0)      &  60              \\  
14  & 1 (0) &        &                &               &                  &                  \\  
16  & 1 (0) & 2  (0) &                &  9 (1)        &                  &                  \\  
18  & 2 (1) &        &   3 (0)        &  1 (0)        &                  &  6               \\  
24  &       &        &                &  7 (3)        &  1      (0)      &  11              \\  
27  &       &        &                &               &                  &  3               \\  
36  &       &        &   1 (0)        &  1 (0)        &                  &  3               \\  
48  &       &        &                &  2 (0)        &                  &                  \\  
54  & 1 (0) &        &                &               &                  &                  \\  
72  &       &        &   1 (0)        &  1 (0)        &                  &  3               \\  
108 & 1 (0) &        &                &               &                  &                  \\  
144 &       &        &                &  1 (0)        &                  &                  \\  
504 & 1 (0) &        &                &               &                  &                  \\  
1008& 1 (1) &        &                &               &                  &                  \\\hline
any & 12 (5)& 244 (0)& 355 (0)        & 107427 (388)  & 1896682 (0)      & 12659522616
\end{tabular}
 \caption{The number of the isomorphism classes of \STS$(21)$ with sub-\TD$(3,6)$, sorted by the number $\tau_6$ of
 sub-\TD$(3,6)$, the number $\sigma_9$ of  sub-\STS$(9)$, and the number of automorphisms
 (the number or nonisomorphic resolvable systems, if known, is given in parenthesis). }
 \label{t:1}
\end{table}

The results of the calculation are reflected in Table~\ref{t:1}.
The last column of the table was calculated by comparing
the data in the other columns with the results of \cite{KOP:2015:subSTS}.
All calculations took few core-hours on a modern PC.
We summarize the results in the following theorem.

\begin{theorem}
There are $2004720$
Steiner triple systems of order $21$ with transversal subdesigns on $3$ groups of size $6$.
$2004109$ of them have exactly one such sub-\TD$(3,6)$, 
$599$ have exactly three  sub-\TD$(3,6)$, and $12$ have seven sub-\TD$(3,6)$
(by Lemma~\ref{l:3flow}, the last group coincides with the $12$ \STS$(21)$
having $7$ sub-\STS$(7)$, found in \cite{KOP:2015:subSTS}).
\end{theorem}


\section{Validity of the results} \label{s:valid}
In this section, we consider a double-counting argument that validates the results of computing.

\begin{proposition}\label{p:pairs}
Given a point set $S$ of size $21$, there are exactly
 $$ \frac{21!}{3! ^2 \cdot 6!^3} \cdot (120^3 + 3\cdot 720\cdot 120^2) \cdot 812851200
 =101473423278637842432000000$$
 pairs $(\bB, \tT)$, where $(S,\bB)$ is an \STS$(21)$ and $\tT$ is a sub-\TD$(3,6)$ of $\bB$.
\end{proposition}
\begin{proof}
We first remind that there are $840$ different \STS$(9)$ with given support, see, e.g., \cite{A030128};
a given triple of points belongs to exactly $120$ of them.

A set of cardinality $21$ can partitioned into a flower $\{A,B,C,D\}$ in
 \mbox{${21!}\cdot {3!}^{-2} \cdot {6!}^{-3}$}
 ways.
 Assuming that $D$ is a block, we can choose an almost-sub-STS with each of the supports
 $A \cup D$, $B \cup D$, $C \cup D$ in $120$ ways.
 Assuming that the $D$ is \emph{not} a block,
 we can choose which of $A \cup D$, $B \cup D$, $C \cup D$
 is the support of a sub-STS in $3$ ways,
 then choose that sub-STS in $840-120=720$ ways,
 then choose each of the remaining two almost-sub-STS's
 in $120$ ways.
 Finally, we choose a transversal design with the groups
 $A$, $B$, $C$ in $812851200$ ways 
 (the total number of different $6\times 6$ latin squares \cite{A002860}).
\end{proof}

On the other hand, we can calculate the same number based on the given representatives of the isomorphism classes.

\begin{proposition}\label{p:pairs2}
Let $S$ be a set of $21$ points and let $\mathbf S$ be a set of representatives of all isomorphism
classes of \STS$(21)$ on $S$. The number of
pairs $(\bB, \tT)$ where $(S,\bB)$ is an \STS$(21)$ and $\tT$ is a sub-\TD$(3,6)$ of $\bB$
is calculated by the formula
\begin{equation}\label{eq:sum}
 \sum_{\bB \in \mathbf S} N(\bB)\cdot \frac{21!}{|\mathrm{Aut(\bB)}|},
\end{equation}
where $N(\bB)$ is the number of sub-\TD$(3,6)$ in $\bB$.
\end{proposition}
Using the data in Table~\ref{t:1}, we can compute the nonzero (with $N(\bB)>0$) terms in the sum \eqref{eq:sum},
which happens to coincide with the value in Proposition~\ref{p:pairs}.
This approves the results of our computing.

\section{Resolvability}\label{s:resol}
A Steiner triple system $(S,\bB)$ is called
\emph{resolvable} if $\bB$ can be partitioned into parallel classes,
where a \emph{parallel class} is a partition of $S$ into blocks.
We check all found systems on resolvability and found $393$
isomorphism classes of resolvable STS of considered type.
As we see from Table~\ref{t:1},
there is no resolvable STS$(21)$ with sub-\TD$(3,6)$
and only one sub-\STS$(9)$.
We can prove this fact theoretically.

\begin{proposition}\label{th:resol}
 If an \STS$(21)$ 
 has a sub-\TD$(3,6)$
 and only one sub-\STS$(9)$ 
 then $\dD$ is not resolvable.
\end{proposition}
\begin{proof}
Let $(S,\dD)$ be an STS, 
 let $(A\cup B\cup C,\{A,B,C\},\tT)$ be a sub-\TD$(3,6)$,
 corresponding to the flower $\{A,B,C,D\}$ 
 and let $\aA$ be the unique sub-STS.
 Without loss of generality we assume
 that the support of $\aA$ is $A\cup D$.
 So, there are two almost-sub-STS with the supports 
 $B\cup D$ and $C\cup D$
 respectively, and the missing triple $D$.
 By the hypothesis, $D \not\in \dD$.

Let $D=\{a,b,c\}$. Seeking a contradiction, assume that there is a resolution.
Consider the block $U$ containing $a$ and $b$ and consider the parallel class $\pP$
containing this block. Denote $t:=|\pP \cap \tT|$.
We state that

(*) \emph{the block $V$ from $\pP$ containing $c$ belongs to $\aA$}.
Indeed, if it is in $\bB$, then $|B\backslash V|=4$, and $t$ of these $4$ points
are covered by blocks from $\tT\cap\pP$, the other $4-t$ being covered  by blocks from $\bB\cap\pP$.
Hence, $4-t \equiv 0\bmod 3$. On the other hand,
$t$ of the $6$ points of $C$
are covered by blocks from $\tT\cap\pP$, the other $6-t$ being covered  by blocks from $\cC\cap\pP$.
So, $6-t \equiv 0\bmod 3$, a contradiction. Similarly, $V\not\in \cC$, and (*) holds.

Since $|A\backslash U \backslash V|=3$, we have $t\le 3$.
Therefore, $\pP$ contains at least one block
from $\bB$ and at least one block from $\cC$,
and these blocks have no points in $D$.
The same can be said about the parallel class that contains the block with $a$ and $c$.
And similarly, for the parallel class that contains the block with $b$ and $c$.
We conclude that $\bB$ has at least three blocks disjoint with $D$.
This contradicts to Proposition~\ref{p:sts9}.
\end{proof}

\section{Acknowledgments}
The authors are grateful to Svetlana Topalova for useful discussion.
\bibliographystyle{plain}
\bibliography{../../k}
\end{document}